\documentclass[12pt]{amsart}
\usepackage{amsmath,amsfonts,amssymb,amsthm,amsxtra,latexsym,verbatim,amscd}
\usepackage{eufrak}
\usepackage[all]{xy}

\theoremstyle{plain}
\newtheorem{thm}{Theorem}[section]

\newtheorem{lem}[thm]{Lemma}
\newtheorem{cor}[thm]{Corollary}
\newtheorem*{main}{Main Theorem}

\theoremstyle{definition}
\newtheorem{defn}[thm]{Definition}

\theoremstyle{remark}
\newtheorem{ex}[thm]{Example}
\newtheorem{rmk}[thm]{Remark}

\newcommand{\A}{\mathbb{A}}
\newcommand{\C}{\mathbb{C}}
\newcommand{\N}{\mathbb{N}}

\newcommand{\R}{\mathbb{R}}
\newcommand{\Z}{\mathbb{Z}}

\newcommand{\I}{\mathcal{I}}
\newcommand{\J}{\mathcal{J}}
\renewcommand{\L}{\mathcal{L}}
\renewcommand{\O}{\mathcal{O}}

\renewcommand{\a}{\mathfrak{a}}
\renewcommand{\b}{\mathfrak{b}}

\DeclareMathOperator{\interior}{int}
\DeclareMathOperator{\ord}{ord}

\DeclareMathOperator{\Spec}{Spec}

\begin{document}

\title{Multiplier ideals of monomial space curves}
\author{Howard~M~Thompson}
\address{Department~of~Mathematics, University~of~Michigan~--~Flint, Flint, Michigan 48502-1950}
\email{hmthomps@umflint.edu}
\date{}

\begin{abstract}
  This paper presents a formula for the multiplier ideals of a monomial space curve. The formula is obtained from a careful choice of log resolution. We construct a toric blowup of affine space in such a way that a log resolution of the monomial curve may be constructed from this toric variety in a well controlled manner. The construction exploits a theorem of Gonz\'{a}lez~P\'{e}rez and Teissier~\cite{MR1892938}
\end{abstract}

\maketitle

\section{Introduction}

Multiplier ideals have become an important tool in algebraic geometry. However, they are notoriously difficult to compute. We do have formulas or partial information in a few cases. See Howald~\cite{MR1828466}, Blickle~\cite{MR2092724}, Musta{\c{t}}{\u{a}~\cite{MR2231883}, Saito~\cite{MR2339839}, Teitler~\cite{MR2373586}, Teitler~\cite{MR2324623}, Smith \& Thompson~\cite{MR2389246}, Tucker~\cite{MR2592954}, and Shibuta \& Takagi~\cite{MR2533766}. Let's recall the definition.

\begin{defn}
  Let $X$ be a smooth irreducible variety over $\C$. Let $I\subseteq\O_X$ be a nonzero ideal sheaf and let $\lambda>0$ be a rational number. Fix a log resolution $\mu:\widetilde{X}\to X$ of $I$ and let $F$ be the effective divisor such that $I\cdot\O_X=\O_{\widetilde{X}}(-F)$. The \emph{multiplier ideal} $\J(I^{\lambda})\subseteq\O_X$ associated to $\lambda$ and $I$ is defined to be
  \[
    \J(I^{\lambda})=\mu_*\O_{\widetilde{X}}(K_{\widetilde{X}/X}-\lfloor\lambda\cdot F\rfloor).
  \]
\end{defn}
The multiplier ideal does not depend upon the choice of log resolution. If, in addition, we assume $X=\Spec R$ is affine, we also have the following formula, essentially due to Lipman~\cite{MR1306018}:
\begin{equation} \label{Lipman}
  \J(I^{\lambda})=\bigcap_{\nu}\{f\in R\mid\nu(f)\geq\lfloor\nu(\a)\lambda-\nu(J_{R_{\nu}/R})\rfloor\}.
\end{equation}
where the intersection is over all divisorial valuations $\nu$ such that $R_\nu$ dominates $R$. Note that it suffices to take the intersection over the valuations whose associated divisors appear on a single log resolution.

Given any integer vector $\mathbf{m}=\begin{bmatrix}m_1&m_2&m_3\end{bmatrix}$ in the nonnegative octant of $\R^3$, let $\ord_\mathbf{m}$ be the monomial valuation on $\C[x,y,z]$ given by $x\mapsto m_1$, $y\mapsto m_2$ and $z\mapsto m_3$. Given the ideal $I\subseteq\C[x,y,z]$ of a monomial space curve $\{(t^{n_1},t^{n_2},t^{n_3})\mid t\in\C\}\subseteq\A^3$, let $\tau$ be the smallest monomial ideal containing $I$, let $\mathbf{n}=\begin{bmatrix}n_1&n_2&n_3\end{bmatrix}$ and let $\{f_1,f_2,f_3,\ldots\}$ be a set of binomial generators for $I$  written in increasing order of vanishing with respect to $\ord_\mathbf{n}$. We will identify a finite set of primitive lattice vectors $G$ in the nonnegative octant of $\R^3$ such that
\begin{main}
\[
  \J(I^\lambda)=I^{(\lfloor\lambda-1\rfloor)}\cap\J(\tau^\lambda)\cap \\
  \bigcap_{\mathbf{m}\in G}\{f\mid\ord_\mathbf{m}(f)\geq\lfloor\ord_\mathbf{m} (f_2)\lambda-k_\mathbf{m}\rfloor\}
\]
where $k_\mathbf{m}=m_1+m_2+m_3-1+\ord_\mathbf{m}(f_2)-\ord_\mathbf{m}(f_1)$.
\end{main}
In the next section, we will recall what we know about formulas like Lipman's in special cases. And, in the final section we will recall the result of Gonz\'{a}lez~P\'{e}rez and Teissier~\cite{MR1892938} and use it to prove our theorem

\section{On formulas for multiplier ideals}

If $R$ is a regular essentially finitely generated $\C$-algebra, $I\subseteq R$ is a prime ideal of height $c$ and $R/I$ is regular, then $\J(I^\lambda)=I^{(\lfloor\lambda-c+1\rfloor)}$ is a symbolic power of $I$. That is, when $I$ is the ideal of a smooth subvariety, it suffices to consider only the valuation corresponding to the blowup of $I$.

When $R$ is a polynomial ring and $\a$ is a monomial ideal, we have the following theorem.
\begin{thm}[Howald's Theorem \cite{MR1828466}]
  Let $I\subseteq\C[x_1,\ldots,x_r]$, let $P(I)\subseteq\R^r$ be the Newton polyhedron of $I$, let $\interior(\lambda\cdot P(I))$ be the interior of the dilated polyhedron, and let $\mathbf{1}=(1,\ldots,1)\in\N^r$. Then, the multiplier ideal $\J(I^\lambda)$ is the monomial ideal
  \[
    \J(I^\lambda)=(\mathbf{x}^{\mathbf{v}}\mid\mathbf{v}+\mathbf{1}\in\interior(\lambda\cdot P(I))).
  \]
\end{thm}
This theorem tells us that it suffices to consider the divisors that appear on the blowup of the monomial ideal and it gives us a convenient way to compute the intersection \eqref{Lipman} in the monomial case. Recall that the term ideal of an ideal $I$ is the smallest monomial ideal $\tau$ containing $I$. If $I\subseteq\C[x_1,\ldots,x_r]$ is a prime of height $c$ and $\tau$ is the term ideal of $I$,
\[
  \J(I^\lambda)\subseteq I^{(\lfloor\lambda-c+1\rfloor)}\cap\J(\tau^\lambda)
\]
since the formation of multiplier ideals respects containment. So, to give a formula for the multiplier ideals of an ideal $I$, it suffices to find a convenient finite set $N$ such that
\[
  \J(I^\lambda)=I^{(\lfloor\lambda-c+1\rfloor)}\cap\J(\tau^\lambda)\cap\bigcap_{\nu\in N}\{f\in R\mid\nu(f)\geq\lfloor\nu(I)\lambda-\nu(J_{R_{\nu}/R})\rfloor\}
\]
and a convenient way to compute the numbers $\nu(I)$ and $\nu(J_{R_{\nu}/R})$ for $\nu\in N$. Next we will state a result that gives insight into which valuations $N$ might contain. To do this, we recall the notion of a strong factorizing desingularization of a subscheme of a smooth variety.
\begin{thm}[of strong factorizing desingularization \cite{MR1971154}]
  Let $X$ be a smooth variety over a field $\C$ of characteristic zero, let $Z\subset X$ be a closed subvariety, and let $I$ be the corresponding sheaf of ideals. Then there is a sequence of blowups with smooth centers whose composition is
  \[
    \mu:\widetilde{X}\to X
  \]
  so that if $\widetilde{Z}\subset\widetilde{X}$ is the strict transform of $Z$ and $E$ is the exceptional locus of $\mu$, we have.
  \begin{itemize}
    \item[(1)] $E$ is a divisor with simple normal crossings support.
    \item[(2)] At each step the blowup center is disjoint from the regular locus of the strict transform.
    \item[(3)] $\widetilde{Z}$ is smooth and has simple normal crossings with $E$.
    \item[(4)] $I\cdot\O_{\widetilde{X}}=\L\cdot\I_{\widetilde{Z}}$ where $\I_{\widetilde{Z}}$ is the ideal sheaf of $\widetilde{Z}$ and $\L=\O_{\widetilde{X}}(-E)$.
  \end{itemize}
\end{thm}
I've stated a weakening of the theorem that is convenient for our purposes. Such a desingularization exists if $Z$ is a subscheme whose regular locus is dense. Without the fourth condition, we would have Hironaka's resolution of singularities. Here is an example to illustrate the difference and give some geometric intuition into this fourth condition.
\begin{ex}
  Let $X=\Spec\C[x,y,z]$ be $\A^3$ and let $Z$ be the union of the $x$ and $y$ axes. Then $I=(xy,z)$ and the blowup of the origin is an embedded resolution of singularities that is not strong factorizing. Consider the chart where the ideal of the exceptional divisor is $(x)$.
  \[
    I\cdot\C\left[x,\frac{y}{x},\frac{z}{x}\right]=\left(x^2\frac{y}{x},x\frac{z}{x}\right)=(x)\left(x\frac{y}{x},\frac{z}{x}\right)
  \]
  but the ideal of the strict transform is $\left(\frac{y}{x},\frac{z}{x}\right)$. The problem here is that the pullback of $I$ to the blowup consists of the exceptional divisor, the strict transform which is two disjoint lines, and an embedded line on the exceptional divisor connecting the two points where the strict transform meets the exceptional divisor:
  \begin{align*}
    I\cdot\C\left[x,\frac{y}{x},\frac{z}{x}\right]&=(x)\cap\left(\frac{y}{x},\frac{z}{x}\right)\cap\left(x^2,\frac{z}{x}\right) \\
    I\cdot\C\left[\frac{x}{y},y,\frac{z}{y}\right]&=(y)\cap\left(\frac{x}{y},\frac{z}{y}\right)\cap\left(y^2,\frac{z}{y}\right) \\
    I\cdot\C\left[\frac{x}{z},\frac{y}{z},z\right]&=(z)
  \end{align*}
  After blowing up this embedded line, we get a strong factorizing resolution:
  \begin{align*}
    I\cdot\C\left[x,\frac{y}{x},\frac{z}{x^2}\right]&=(x)^2\cdot\left(\frac{y}{x},\frac{z}{x^2}\right) \\
    I\cdot\C\left[\frac{x^2}{z},\frac{y}{x},\frac{z}{x}\right]&=\left(\frac{x^2}{z}\right)\cdot\left(\frac{z}{x}\right)^2 \\
    I\cdot\C\left[\frac{x}{y},y,\frac{z}{y^2}\right]&=(y)^2\cdot\left(\frac{x}{y},\frac{z}{y^2}\right) \\
    I\cdot\C\left[\frac{x}{y},\frac{y^2}{z},\frac{z}{y}\right]&=\left(\frac{y^2}{z}\right)\cdot\left(\frac{z}{y}\right)^2 \\
    I\cdot\C\left[\frac{x}{z},\frac{y}{z},z\right]&=(z)
  \end{align*}
  Here the pullback of $I$ is one copy of the first exceptional divisor, two copies of the second exceptional divisor, and the strict transform.
\end{ex}
\begin{thm}
  Let $I\subseteq\C[x_1,\ldots,x_r]$ be a prime of height $c$, let $Z=Z(I)\subseteq X$ be the subvariety cut out by $I$, let $\mu:\widetilde{X}\to X$ be a strong factorizing desingularization of $Z$, let $\widetilde{Z}$ be the strict transform of $Z$, let $\L$ be the invertible sheaf such that $I\cdot\O_{\widetilde{X}}=\L\cdot\I_{\widetilde{Z}}$ and let $\b=\mu_*\L$. Then
  \[
    \J(I^\lambda)=I^{(\lfloor\lambda-c+1\rfloor)}\cap\J(\b^\lambda)
  \]
\end{thm}
\begin{proof}
  Evidently, $I\subseteq\b$. So, the inclusion $\J(I^\lambda)\subseteq I^{(\lfloor\lambda-c+1\rfloor)}\cap\J(\b^\lambda)$ is clear.

  On the other hand, $I\cdot\O_{\widetilde{X}}\subseteq\b\cdot\O_{\widetilde{X}}\subseteq\L$. So, after intersecting with $\I_{\widetilde{Z}}$ we obtain  $I\cdot\O_{\widetilde{X}}\subseteq\I_{\widetilde{Z}}\cap\b\cdot\O_{\widetilde{X}}\subseteq I\cdot\O_{\widetilde{X}}$ since $I\cdot\O_{\widetilde{X}}\subseteq\I_{\widetilde{Z}}$ and $\L\cap\I_{\widetilde{Z}}=\L\cdot\I_{\widetilde{Z}}=I\cdot\O_{\widetilde{X}}$. That is, $(\I_{\widetilde{Z}}\cap\b\cdot\O_{\widetilde{X}})\subseteq I\cdot\O_{\widetilde{X}}$. Now let $\pi:Y\to X$ be a log resolution of $\b$ (along with $\a$ and $\I_{\widetilde{Z}}$ as well) that factors through the blowup of $\I_{\widetilde{Z}}$ on $\widetilde{X}$. Evidently, $I\cdot\O_Y$ differs from $\b\cdot\O_Y$ only along the support of $\I_{\widetilde{Z}}\cdot\O_Y$. Now, pushing down from $Y$ establishes the formula.
\end{proof}
So, multiplier ideals reflect something of the geometry of strong factorizing resolutions. This result tells us that if $\widetilde{X}$ is an embedded resolution of $I$, then we should look to the embedded components of $I\cdot\O_{\widetilde{X}}$ to find the sources of additional divisorial valuations that determine $\J(I^\lambda)$.
\begin{rmk}
  There are many ideals that suffice in the role of $\b$ in the above theorem. And, there is a largest such ideal. I am not sure whether the procedure above picks that largest ideal. Notice that a $\b$-ideal can be chosen so that it is supported on the singular locus of $I$.
\end{rmk}

\section{The formula for monomial space curves}

From now on, let $R=\C[x_1,x_2,x_3]$, let $\mathbf{n}=\begin{bmatrix}n_1&n_2&n_3\end{bmatrix}$ be a primitive positive integer vector, let $I$ be the kernel of the $\C$-algebra homomorphism $\varphi:R\to\C[t]$ given by $x_1\mapsto t^{n_1}$, $x_2\mapsto t^{n_2}$ and $x_3\mapsto t^{n_3}$, let $\tau$ be the term ideal of $I$, let $C\subseteq\A^3$ be the monomial space curve cut out by $I$, and let $\ord_{\mathbf{m}}(x_i)=m_i$ for any nonnegative integer vector $\mathbf{m}=\begin{bmatrix}m_1&m_2&m_3\end{bmatrix}$. Now write the minimal binomial generators $f_1,f_2,\ldots$ of $I$ in order of increasing $\ord_{\mathbf{n}}$ order and write
\[
  f_1=\mathbf{x}^{\mathbf{u}_1}-\mathbf{x}^{\mathbf{v}_1},f_2=\mathbf{x}^{\mathbf{u}_2}-\mathbf{x}^{\mathbf{v}_2},\ldots
\]

We will construct a toric blowup $\mu:X\to\A^3$ such that the strict transform of $C$ is smooth and has normal crossings with the exceptional locus. Moreover, the embedded components of the pullback of $I$ will be smooth curves contained in the smooth locus of $X$ with at most one such curve on each exceptional divisor. Furthermore, an embedded curve will have normal crossings with each torus invariant divisor it meets that does not contain it. Once we have all that, any toric desingularization $\widetilde{X}$ of $X$ will give an embedded resolution of $I$, the order of vanishing of $I$ on any exceptional divisor of $X$ will be the same as that of its term ideal $\tau$, and will will have
\[
  \J(I^\lambda)=I^{(\lfloor\lambda-1\rfloor)}\cap\J(\tau^\lambda)\cap\bigcap_{\nu\in G}\{f\in R\mid\nu(f)\geq\lfloor\nu(\a)\lambda-\nu(J_{R_{\nu}/R})\rfloor\}
\]
where the set $G$ comes from the embedded components on $X$. We will then study these embedded components and as a last step, we will identify the set $G$.

To create $X$, we will exploit a theorem of Gonz\'{a}lez~P\'{e}rez and Teissier.
\begin{thm}[\cite{MR1892938}]
  Let $\Gamma$ be a finitely generated submonoid of $\Z\Gamma\cong\Z^d$. Assume $\Gamma$ has a trivial unit group and let $\varphi:\N^r\to\Gamma$ be a surjective monoid homomorphism. $\varphi$ induces an embedding of (not necessarily normal) toric varieties $Z\to\A^r$ where $Z=\Spec\C[\Gamma]$. Let $\ell$ be the orthogonal complement of the kernel of the linear map $\R^r\to\R^d$ extending $\varphi$, let $\sigma$ be the cone $\R_{\geq0}^r\cap\ell$, let $N_{\sigma}$ be the lattice $\R\sigma\cap N$, and let $Z_{\sigma}$ be the toric variety of the cone $\sigma$ with respect to the lattice $N_{\sigma}$..  Let $\Sigma$ be any fan supported on $\R_{\geq0}^r$ containing the cone $\sigma$ and let $\pi_{\Sigma}:X_{\Sigma}\to\A^r$ be the associated toric modification.
  \begin{itemize}
    \item[(1)] Then the strict transform $\widetilde{Z}$ of $Z$ in $X_{\Sigma}$ is contained in the open affine subvariety $X_{\sigma}$ and the restriction $\pi_{\Sigma}|_{\widetilde{Z}}:\widetilde{Z}\to Z$ is the normalization map. In fact, $\widetilde{Z}\cong Z_{\sigma}$.
    \item[(2)] If $\Sigma'$ is any regular fan refining $\Sigma$, then the map $\pi_{\Sigma'}:X_{\Sigma'}\to\A^r$ is an embedded resolution of $Z$. (In particular, the strict transform of $Z$ via $\pi_{\Sigma'}$ is transverse to the orbit stratification of $X_{\Sigma'}$.)
  \end{itemize}
\end{thm}
We will apply this theorem to the monoid $\Gamma\subseteq\N$ generated by $n_1$, $n_2$ and $n_3$. In this case, the embedded toric variety is our curve $C\subseteq\A^3$ and $\ell$ is the span of $\mathbf{n}$. Let $\Sigma_1$ be the stellar subdivision of $\R_{\geq0}^3$ along the ray $\rho$ with primitive vector $\mathbf{n}$. On $X_{\Sigma_1}$, the strict transform of $C$ is normal and contained in the open affine $X_{\rho}$. But, $C$ is a curve and $\rho$ is a ray. So, the strict transform is smooth and contained in the smooth locus of $X_{\Sigma_1}$. And, any toric desingularization of $X_{\Sigma_1}$ provides an embedded resolution of $C$.

\begin{lem}
  $I\cdot\O_{X_{\Sigma_1}}$ has an embedded component exactly when $\ord_\mathbf{n}(f_1)<\ord_\mathbf{n}(f_2)$ and $f_1$ cuts out the embedded curve on the exceptional divisor set-theoretically in this case.
\end{lem}
\begin{proof}
  Now we study embedded components on $X_{\Sigma_1}$ in order to understand how to build our good toric variety $X$. Let $S$ be the coordinate ring of $X_{\rho}$. Then
  \[
    S=R[(\mathbf{x}^{\mathbf{u}_1-\mathbf{v}_1})^{\pm1},(\mathbf{x}^{\mathbf{u}_2-\mathbf{v}_2})^{\pm1},\ldots],
  \]
  and $IS=(y^{\ord_\mathbf{n}(f_1)}g_1,y^{\ord_\mathbf{n}(f_2)}g_2,\ldots)$ where $(g_1,g_2,\ldots)$ is the ideal of the strict transform and $(y)$ is the prime ideal corresponding to the exceptional divisor. We know that the strict transform is smooth and meets the exceptional divisor transversally. In particular, $IS$ is generated by two elements. When we combine this with the fact that we wrote the generators in order of increasing order of vanishing, we have
  \[
    IS=(y^{\ord_\mathbf{n}(f_1)}g_1,y^{\ord_\mathbf{n}(f_2)}g_2)=(y)^{\ord_\mathbf{n}(f_1)}\cap(g_1,y^{\ord_\mathbf{n}(f_2)})\cap(g_1,g_2)
  \]
  From this calculation, we see that $I\cdot\O_{X_{\Sigma_1}}$ has an embedded component exactly when $\ord_\mathbf{n}(f_1)<\ord_\mathbf{n}(f_2)$ and $f_1$ cuts out the embedded curve on the exceptional divisor set-theoretically in this case.
\end{proof}
\begin{cor}
  If $\ord_\mathbf{n}(f_1)=\ord_\mathbf{n}(f_2)$, then any toric desingularization of $X_{\Sigma_1}$ yields a strong factorizing resolution of $C$. And,
  \[
    \J(I^\lambda)=I^{(\lfloor\lambda-1\rfloor)}\cap\J(\tau^\lambda).
  \]
\end{cor}
\begin{rmk}
  Notice that on any toric desingularization of $X_{\Sigma_1}$, the pullback of $I$ is the same as the pullback of $(f_1,f_2)$.
\end{rmk}
\begin{lem}
  If $\ord_\mathbf{n}(f_1)\neq\ord_\mathbf{n}(f_2)$, then the embedded curve passes through one or two of the three torus fixed points of $X_{\Sigma_1}$.
\end{lem}
\begin{proof}
  Notice that $\Sigma_1$ is the fan of the normalized blowup of the ideal $(x_1^{n_2n_3},x_2^{n_3n_1},x_3^{n_1n_2})$. $X_{\Sigma_1}$ has three torus fixed points. And, at each of these torus fixed points one of the three generators of the ideal $(x_1^{n_2n_3},x_2^{n_3n_1},x_3^{n_1n_2})$ is a principal generator for $(x_1^{n_2n_3},x_2^{n_3n_1},x_3^{n_1n_2})$. Since $f_1=\mathbf{x}^{\mathbf{u}_1}-\mathbf{x}^{\mathbf{v}_1}$, the embedded component passes through the torus fixed closed point where $x_i^\frac{n_1n_2n_3}{n_i}$ generates the ideal $(x_1^{n_2n_3},x_2^{n_3n_1},x_3^{n_1n_2})$ exactly when neither $\mathbf{x}^{\mathbf{u}_1}$ nor $\mathbf{x}^{\mathbf{v}_1}$ is a power of $x_i$. Moreover, since there are only three variables and $f_1$ is irreducible, either one or two of the terms of $f_1$ is a power of a variable.
\end{proof}
Let $Y$ be the toric surface containing $C$ cut out by $(f_1)$. We will apply the theorem of Gonz\'{a}lez~P\'{e}rez and Teissier to the semigroup of the surface $Y$ as well. Here the kernel of the extension of $\varphi$ is spanned by $\mathbf{u}_1-\mathbf{v}_1$.To create the fan $\Sigma_2$, slice the fan ${\Sigma_1}$ with the hyperplane, $H$, orthogonal to $\mathbf{u}_1-\mathbf{v}_1$. This is equivalent blowing up $(\mathbf{x}^{\mathbf{u}_1},\mathbf{x}^{\mathbf{v}_1})$ the term ideal of $f_1$. 
\begin{lem}
   Let $\Sigma$ be any regular fan refining $\Sigma_2$.
  \begin{itemize}
    \item[(1)] $I\cdot\O_{X_{\Sigma_2}}$ has no embedded components on the newly created divisors.
    \item[(2)] The embedded component of $I\cdot\O_{X_{\Sigma_2}}$ does not pass through any torus fixed closed points.
    \item[(3)] The strict transform $C$ in $X_{\Sigma}$ is smooth and contained in the smooth locus of $X_{\Sigma}$.
    \item[(4)] The map $\pi_{\Sigma}:X_{\Sigma}\to\A^r$ is a simultaneous embedded resolution of both $C$ and $Y$. In particular, the strict transform of $Y$ via $\pi_{\Sigma}$ is transverse to the orbit stratification of $X_{\Sigma}$.
    \item[(5)] The embedded components of $I\cdot\O_{X_{\Sigma}}$ are supported on the intersection of $Y$ with the torus invariant exceptional divisors of $X$. In particular, the supports of the embedded components are smooth curves and each of these curves meets any invariant divisor that does not contain it transversally.
  \end{itemize}
\end{lem}
\begin{proof}
  Note that at most two new rays are created by cutting with $H$ and each of these rays is on at least one coordinate hyperplane. So, at least one of the $x_i$ does not vanish on the corresponding divisor, However, $I$ must contain a binomial one of whose terms is a power of $x_i$. Therefore, the order of vanishing of $I$ along the divisor must be zero. Hence, $I\cdot\O_{X_{\Sigma_2}}$ has no embedded components on the newly created divisors. Since $f_1$ factors as $\mathbf{x}^{\mathbf{u}_1}(1-\mathbf{x}^{\mathbf{v}_1-\mathbf{u}_1})$ or $\mathbf{x}^{\mathbf{v}_1}(\mathbf{x}^{\mathbf{u}_1-\mathbf{v}_1}-1)$ on each toric open affine chart and the factor $1-\mathbf{x}^{\mathbf{v}_1-\mathbf{u}_1}$ (or $\mathbf{x}^{\mathbf{u}_1-\mathbf{v}_1}-1$) is a unit in the local ring of a torus fixed closed point, after this blowup the embedded component no longer passes through any of the torus fixed closed points. So, on $X_{\Sigma_2}$ the embedded component meets only torus invariant curves. Moreover, looking at the factorization of $f_1$ is one way to see the embedded component is smooth away from these curves and meets these curves as nicely as possible at the identity of the one-dimensional tori of these curves. The rest follows immediately from the theorem of Gonz\'{a}lez~P\'{e}rez and Teissier.
\end{proof}
If these one-dimensional tori are in the smooth locus of $X_{\Sigma_2}$, set $X=X_{\Sigma_2}$. If one or more of these one-dimensional tori are not in the smooth locus, the third and final step in creating $X$ is to desingularize each of corresponding the two-dimensional cones in the fan of $X_{\Sigma_2}$ that correspond to the torus invariant curves that meet the embedded component. Use the minimal toric desingularization(s). Call the resultant fan $\Sigma_3$. We have created $X=X_{\Sigma_3}$. Now we will study the embedded components of $I\cdot\O_X$.
\begin{lem}
  After blowing up the support of this embedded component on the divisor corresonding to the ray with primite vector $\mathbf{m}=\begin{bmatrix}m_1 & m_2 & m_3\end{bmatrix}$ and the supports of its strict transforms $\ord_{\mathbf{m}}(f_2)-\ord_{\mathbf{m}}(f_1)$ times for each ray of $\Sigma_3$ and resolving the toric singularities of the ambient space, we are left with a strong factorizing resolution of $I$.
\end{lem}
\begin{proof}
Note that all the rays of fan $\Sigma_3$ other than the positive coordinate axes lie on the hyperplane $H$. If $\mathbf{m}=\begin{bmatrix}m_1 & m_2 & m_3\end{bmatrix}$ is a primitive vector along such a ray, then the orders of vanishing of $f_1$ and $f_2$ along the corresponding exceptional divisor are
\[
  \ord_\mathbf{m}(f_1)=\langle\mathbf{m},\mathbf{u_1}\rangle=\langle\mathbf{m},\mathbf{v_1}\rangle\text{ and }\ord_\mathbf{m}(f_2)=\min(\langle\mathbf{m},\mathbf{u_2}\rangle,\langle\mathbf{m},\mathbf{v_2}\rangle).
\]
Note that the corresponding divisor contains an embedded component exactly when $\ord_{\mathbf{m}}(f_1)<\ord_{\mathbf{m}}(f_2)$. And, recall that all the embedded components are smooth curves, at most one on each invariant divisor, meetng the invariant divisors that don't contain them transversally. So, after blowing up the support of this embedded component and the supports of its strict transforms $\ord_{\mathbf{m}}(f_2)-\ord_{\mathbf{m}}(f_1)$ times there is no embedded component. Moreover, after following this procedure for each embedded component on $X$ and resolving the toric singularities we are left with a strong factorizing resolution of $I$.
\end{proof}
\begin{lem}
  Fix an $\mathbf{m}$ as above. Let $\nu_0=\ord_\mathbf{m}$ and for $k=1$ to $\ord_{\mathbf{m}}(f_2)-\ord_{\mathbf{m}}(f_1)$, let $\nu_k$ be the discrete valuation corresponding to the $k$-th blowup in a chain of blowups starting at the divisor corresponding to $\mathbf{m}$ as above, then $\nu_k$ is given by the generating sequence $x_1\mapsto m_1$, $x_2\mapsto m_2$, $x_3\mapsto m_3$ and $f_1\mapsto\ord_\mathbf{m}(f_1)+k$. Moreover, for $0\leq k<\ord_{\mathbf{m}}(f_2)-\ord_{\mathbf{m}}(f_1)$,
  \begin{multline*}
    \{f\in R\mid\nu_{k+1}(f)\geq\lfloor\nu_{k+1}(\a)\lambda-\nu_{k+1}(J_{R_{\nu_{k+1}}/R})\rfloor\} \\
    \subseteq \{f\in R\mid\nu_k(f)\geq\lfloor\nu_k(\a)\lambda-\nu_k(J_{R_{\nu_k}/R})\rfloor\}.
  \end{multline*}
\end{lem}
\begin{proof}
  The generating statement is evident. It suffices to find the multiplier $\lambda$ at which the valuation $\nu_k$ excludes the ring element $x_1^ax_2^bx_3^cf_1^d$. Using H\"{u}bl and Swanson~\cite[Lemma~3.4]{MR2492462}, we see
  \[
    \nu_k(J_{R_{\nu_k}/R})=m_1+m_2+m_3-1+k
  \]
  So, we examine the equation
  \[
    am_1+bm_2+cm_3+d(\ord_\mathbf{m}(f_1)+k)+1=(\ord_\mathbf{m}(f_1)+k)\lambda-(m_1+m_2+m_3-1+k).
  \]
  Solving this equation for $\lambda$, we obtain
  \begin{align*}
    \lambda&=\frac{(a+1)m_1+(b+1)m_2+(c+1)m_3+k}{\ord_\mathbf{m}(f_1)+k}+d \\
    &=\frac{(a+1)m_1+(b+1)m_2+(c+1)m_3-\ord_\mathbf{m}(f_1)+\ord_\mathbf{m}(f_1)+k}{\ord_\mathbf{m}(f_1)+k}+d \\
    &=\frac{(a+1)m_1+(b+1)m_2+(c+1)m_3-\ord_\mathbf{m}(f_1)}{\ord_\mathbf{m}(f_1)+k}+d+1
  \end{align*}
  Evidently, this number decreases as $k$ increases.
\end{proof}
To sum up what we know, recall that $I$ is the ideal of a monomial space curve and $\tau$ is its term ideal. Let $\sigma$ be the two-dimensional cone obtained by intersecting the hyperplane orthogonal to $\mathbf{u}_1-\mathbf{v}_1$ with $\R_{\geq0}^3$, let $\sigma_{\mathbf{u}_2}$ and $\sigma_{\mathbf{v}_2}$ be the two maximal cones of the fan obtained by subdividing $\sigma$ along $\mathbf{n}$. More specifically, let  $\sigma_{\mathbf{u}_2}$ be the cone where $\min(\langle\mathbf{m},\mathbf{u_2}\rangle,\langle\mathbf{m},\mathbf{v_2}\rangle)=\langle\mathbf{m},\mathbf{u_2}\rangle$ and let  $\sigma_{\mathbf{v}_2}$ be the cone where $\min(\langle\mathbf{m},\mathbf{u_2}\rangle,\langle\mathbf{m},\mathbf{v_2}\rangle)=\langle\mathbf{m},\mathbf{v_2}\rangle$.  Let $G_{\mathbf{u}_2}$ be the minimal generating set of the monoid $\sigma_{\mathbf{u}_2}\cap\Z^3$, let $G_{\mathbf{v}_2}$ be the minimal generating set of the monoid $\sigma_{\mathbf{v}_2}\cap\Z^3$, and let $G$ be the subset of $G_{\mathbf{u}_2}\cup G_{\mathbf{v}_2}$ contained in the open subcone $\R_{>0}\rho_{\mathbf{u}_2}+\R_{>0}\rho_{\mathbf{v}_2}$ of $\sigma$ with rays $\rho_{\mathbf{u}_2}\subseteq\sigma_{\mathbf{u}_2}$ orthogonal to $\mathbf{u}_1-\mathbf{u}_2$ (or equivalently orthogonal to $\mathbf{v}_1-\mathbf{u}_2$) and $\rho_{\mathbf{v}_2}\subseteq\sigma_{\mathbf{v}_2}$ orthogonal to $\mathbf{v}_1-\mathbf{v}_2$ (or equivalently orthogonal to $\mathbf{u}_1-\mathbf{v}_2$). $G$ is the set of $\mathbf{m}$ such that $\mathbf{m}$ corresponds to a divisor on $X$ that contains an embedded component. For each $\mathbf{m}\in G$, let $\nu_\mathbf{m}$ be  the discrete valuation given by the generating sequence  $x_1\mapsto m_1$, $x_2\mapsto m_2$, $x_3\mapsto m_3$ and $f_1\mapsto\ord_\mathbf{m}(f_2)$.
\begin{main}
  With the notation above,
  \[
    \J(I^\lambda)=I^{(\lfloor\lambda-1\rfloor)}\cap\J(\tau^\lambda)\cap \\
    \bigcap_{\mathbf{m}\in G}\{f\mid\ord_\mathbf{m}(f)\geq\lfloor\ord_\mathbf{m} (f_2)\lambda-k_\mathbf{m}\rfloor\}
  \]
  where $k_\mathbf{m}=m_1+m_2+m_3-1+\ord_\mathbf{m}(f_2)-\ord_\mathbf{m}(f_1)$.
\end{main}
\begin{ex}[the $(3,4,5)$ monomial curve]
  Let $\varphi:\C[x,y,z]\to\C[t]$ be given by $x\mapsto t^3$, $y\mapsto t^4$ and $z\mapsto t^5$. Then $I=(y^2-xz,x^3-yz,z^2-x^2y)$ and $\tau=(y^2,xz,x^3,yz,z^2,x^2y)$. To get $\Sigma_1$, we subdivide along the ray $\begin{bmatrix}3&4&5\end{bmatrix}$. To get $\Sigma_2$, we slice with the hyperplane orthogonal to $\begin{bmatrix}1&-2&1\end{bmatrix}$. So,
  \begin{align*}
    \sigma          &=\left\langle\begin{bmatrix}2&1&0\end{bmatrix},\begin{bmatrix}0&1&2\end{bmatrix}\right\rangle, \\
    \sigma_{x^3}    &=\left\langle\begin{bmatrix}3&4&5\end{bmatrix},\begin{bmatrix}0&1&2\end{bmatrix}\right\rangle, \\
    \sigma_{yz}     &=\left\langle\begin{bmatrix}2&1&0\end{bmatrix},\begin{bmatrix}3&4&5\end{bmatrix}\right\rangle, \\
    G_{x^3}         &=\left\{\begin{bmatrix}3&4&5\end{bmatrix},\begin{bmatrix}2&3&4\end{bmatrix},\begin{bmatrix}1&2&3\end{bmatrix},\begin{bmatrix}0&1&2\end{bmatrix}\right\}, \\
    G_{yz}          &=\left\{\begin{bmatrix}2&1&0\end{bmatrix},\begin{bmatrix}1&1&1\end{bmatrix},\begin{bmatrix}3&4&5\end{bmatrix}\right\}, \\
    \rho_{x^3}      &=\left\langle\begin{bmatrix}2&3&4\end{bmatrix}\right\rangle, \\
    \rho_{yz}       &=\left\langle\begin{bmatrix}1&1&1\end{bmatrix}\right\rangle, \\
    G               &=\left\{\begin{bmatrix}3&4&5\end{bmatrix}\right\},\text{ and} \\
    \J(I^\lambda)  &=I^{(\lfloor\lambda-1\rfloor)}\cap\J(\tau^\lambda)\cap\{f\mid\nu(f)\geq\lfloor9\lambda-12\rfloor\} \\
                    &\text{where }\nu(x)=3,\nu(y)=4,\nu(z)=5\text{ and }\nu(y^2-xz)=9.
  \end{align*}
\end{ex}


\providecommand{\bysame}{\leavevmode\hbox to3em{\hrulefill}\thinspace}
\providecommand{\MR}{\relax\ifhmode\unskip\space\fi MR }
\providecommand{\MRhref}[2]{%
  \href{http://www.ams.org/mathscinet-getitem?mr=#1}{#2}
}
\providecommand{\href}[2]{#2}

\end{document}